\documentclass[12pt,oneside]{amsart}
\usepackage{amsfonts, amsthm, amsmath, amssymb, ,mathtools} 
\usepackage[backend=bibtex, style=alphabetic]{biblatex} 
\usepackage[colorlinks=true, citecolor=blue, backref=none]{hyperref}  
\usepackage{bbm} 
\usepackage{upgreek} 
\usepackage{url}
\usepackage[hyphenbreaks]{breakurl}
\usepackage{tikz-cd} 
\usepackage[scr=boondox,cal=cm]{mathalpha} 
\usepackage[layout=a4paper,right=2cm,left=2cm]{geometry}
\numberwithin{equation}{section}
\usepackage{enumitem} 
\newcounter{dummy}
\makeatletter
\newcommand\myitem[1][]{\item[#1]\refstepcounter{dummy}\def\@currentlabel{#1}}
\makeatother
\bibliography{biblio}

\theoremstyle{plain}
\newtheorem{theorem}{Theorem}[section]
\newtheorem{proposition}[theorem]{Proposition}
\newtheorem{corollary}[theorem]{Corollary}
\newtheorem{lemma}[theorem]{Lemma}
\newtheorem{definition}[theorem]{Definition}

\theoremstyle{remark}
\newtheorem{remark}[theorem]{Remark}

\newtheorem*{acknowledgment}{Acknowledgments}

\newcommand{\Z}{\mathbb{Z}} 
\newcommand{\Q}{\mathbb{Q}} 
\newcommand{\OO}[2]{\mathcal{O}^{#2}_{#1}} 
\newcommand{\cl}[2]{h^{#2}_{#1}} 
\newcommand{\units}[2]{\left(\OO{#1}{#2}\right)^\times} 
\newcommand{\tr}[1]{Tr_{#1}} 
\newcommand{\diff}[1]{\mathfrak{D}_{#1}} 
\newcommand{\cond}[1]{\mathfrak{f}_{#1}} 
\newcommand{\disc}[1]{\mathfrak{d}_{#1}} 
\DeclarePairedDelimiterXPP{\hht}[3]{\widehat{h}^{#1}}{(}{)}{}{#2,#3} 
\DeclarePairedDelimiterXPP{\hh}[3]{h^{#1}}{(}{)}{}{#2,#3} 
\DeclarePairedDelimiterXPP{\HHT}[3]{\widehat{H}^{#1}}{(}{)}{}{#2,#3} 
\DeclarePairedDelimiterXPP{\HH}[3]{H^{#1}}{(}{)}{}{#2,#3} 
\DeclarePairedDelimiterXPP{\HHE}[3]{H_{\acute{e}t}^{#1}}{(}{)}{}{#2,#3} 
\DeclarePairedDelimiterXPP{\HHM}[3]{H_{M}^{#1}}{(}{)}{}{#2,#3} 
\newcommand{\Ext}[4]{\mathrm{Ext}^{#1}_{#2}\left(#3,#4\right)} 
\newcommand{\hq}[2]{\vartheta(#1,#2)} 
\newcommand{\rp}[2][E]{#1(#2)} 
\newcommand{\rpr}[2][E]{\rk{#1(#2)}} 
\newcommand{\tn}[2][E]{C(#1/#2)} 
\newcommand{\rk}[2][\Z]{\mathrm{rk}_{#1}\left(#2\right)} 
\DeclareFontFamily{U}{wncy}{}
\DeclareFontShape{U}{wncy}{m}{n}{<->wncyr10}{}
\DeclareSymbolFont{mcy}{U}{wncy}{m}{n}
\DeclareMathSymbol{\Sh}{\mathord}{mcy}{"58}
\newcommand{\burnside}[2]{B_#1(#2)} 
\newcommand{\repring}[2]{R_#1(#2)} 
\newcommand{\ts}[2][E]{\Sh(#1/#2)} 
\newcommand{\bexp}[1]{\upgamma_{#1}} 
\newcommand{\Hom}{\mathrm{Hom}} 
\newcommand{\Ker}{\mathrm{Ker}} 
\newcommand{\Coker}{\mathrm{Coker}} 
\newcommand{\dual}[1]{#1^*} 
\newcommand{\ddual}[1]{#1^{**}} 
\newcommand{\pdual}[1]{#1^\vee} 
\newcommand{\tors}[1]{#1_{\mathrm{tors}}} 
\newcommand{\mt}[1]{\overline{#1}} 
\newcommand{\rc}[3][]{\mathcal{C}^{#1}_{#2}\!\left(#3\right)} 
\newcommand{\qindex}[1]{q\left(#1\right)} 
\newcommand{\card}[1]{\left|#1\right|} 
 
\newcommand{\kdt}[3]{\widehat{\psi}^{#1}\left(#2,#3\right)} 
 
\newcommand{\roots}[1]{\mu_{#1}} 
\newcommand{\defect}[2]{\beta_{#1}(#2)} 
\newcommand{\reg}[2]{R_{#1}^{#2}} 

\newcommand{\rcfieldext}{\mathscr{L}}
\newcommand{\Rho}{\mathrm{P}} 

\newcommand{\ie}{i.~e.~} 

\begin{document}
\title{\textsc{Regulator Constants and Cohomology}}
\author{Luca Caputo}
\address{Plaza San Nicol\'as 1, 28013 Madrid, Spain}
\email{lc.lucacaputo@gmail.com}
\subjclass[2020]{Primary: 11R33; Secondary: 20C10, 11R34, 11R29, 11R20, 11R70, 11G05}
\begin{abstract}
We show how regulator constants of a finitely generated $\Z[G]$-module can be related to  $G$-cohomology, where $G$ is a finite group. We then derive consequences of this relation for modules naturally arising in number theory, such as rings of integers and units of number fields, $K$-theory groups of rings of integers and Mordell-Weil groups of elliptic curves.
\end{abstract}

\maketitle
\section{Introduction}
Let $G$ be a finite group. Recall that a formal sum $\Theta= \sum_{H\leq G}n_H H$, $n_H\in \Z$, is a $G$-Brauer relation (or simply a Brauer relation) if 
\[
\bigoplus_{n_H >0} \Q[G/H]^{n_H} \cong \bigoplus_{n_H <0} \Q[G/H]^{n_H}
\]  
as $\Q[G]$-modules. The $\Theta$-regulator constant $\rc{\Theta}{M}\in \Q^\times$ of a finitely generated $\Z[G]$-module $M$ is an invariant, introduced by T. and V. Dokchitser in \cite{DokchitserDokchitser2010}, which is essentially a product of determinants of a bilinear form on submodules of fixed elements $M^H$, raised to the power $n_H$, as $H$ varies over the subgroups of $G$.  The purpose of this paper is to investigate the relation between $\rc{\Theta}{M}$ and the cohomology of $M$.

Since regulator constants are defined in terms of submodules of fixed elements $M^H$, it is natural to expect some connection with cohomology, which is the derived functor of $M \mapsto M^H$. For instance cohomology appears when measuring the defect of regulator constants from being multiplicative in exact sequences (see \cite[Proposition 2.12]{Caputo2025}). An even more explicit connection between regulator constants and cohomology appears in the following arithmetic context. Let $L/k$ be a $G$-Galois extension of number fields. When $G$ is the dihedral group $D=D_{2q}$ of order $2q$ with $q$ odd, putting together \cite[Proposition 2.15]{Bartel2012} and \cite[Theorem 3.14]{CaputoNuccio2020}, one can easily get the following formula
\begin{equation}\label{eq:rccohomunitsdihedral}
\rc{\Theta_{D}}{\mathcal{O}_L^\times} = \left(\hht{0}{\Theta_{D}}{\mathcal{O}_L^\times}\hht{1}{\Theta_{D}}{\mathcal{O}_L^\times}\right)^{-1}.
\end{equation}
Here $\mathcal{O}_L^\times$ is the group of units of $L$ and 
\[
\Theta_{D} = 1 + 2D - \Rho - 2\Sigma
\]
is a $D$-Brauer relation, where $\Rho \triangleleft D$ is the subgroup of index $2$ and $\Sigma<D$ is any subgroup of order $2$. Moreover, for $i\in\Z$ and $H\leq G$, $\hht{i}{H}{M}$ is the order of the $i$-th Tate cohomology group $\HHT{i}{H}{M}$ of $H$ with values in $M$ and
\[
\hht{i}{\Theta}{M} = \prod_{H\leq G} \hht{i}{H}{M}^{n_H}.
\]

The main result of this paper is a generalization of \eqref{eq:rccohomunitsdihedral} to any group $G$, any finitely generated $\Z[G]$-module $M$ and any $G$-Brauer relation $\Theta$. If $M$ has no $\Z$-torsion, this generalization reads
\begin{equation}\label{eq:rccohomgeneraltf}
\rc{\Theta}{M}^{2m} = \rc{\Theta}{V}^{-1}\left(\hht{0}{\Theta}{M}\hht{1}{\Theta}{M}\right)^{-2m}
\end{equation}
where $V$ is a finite $G$-module which is self-dual, \ie $V$ is isomorphic to $\pdual{V}=\Hom(V,\Q/\Z)$ as $\Z[G]$-modules, and $m$ is a positive integer (see Proposition \ref{cor:rcformulatf}). For modules with $\Z$-torsion the formula has some additional cohomological factors.  

The proof of \eqref{eq:rccohomgeneraltf}, which is given in Section \ref{sec:main}, is obtained by studying the regulator constant of the $\Z$-dual $\dual{M}= \Hom(M,\Z)$ of $M$. Exploiting an alternative definition of the regulator constants due to Bartel and de Smit (see \cite[Lemma 3.1]{BarteldeSmit2013}), we first show that
\begin{equation}\label{eq:introrc1}
\rc{\Theta}{M}\rc{\Theta}{\dual{M}} = \hht{0}{\Theta}{M}^{-2}.
\end{equation}
To get \eqref{eq:rccohomgeneraltf} we then show that the following equality also holds 
\begin{equation}\label{eq:introrc2}
\rc{\Theta}{\dual{M}}^m = \rc{\Theta}{M}^m\rc{\Theta}{V}\hht{1}{\Theta}{M}^{2m}.
\end{equation}
where $m$ and $V$ are as above. The module $V$, which somehow accounts for the distance of $M$ from being a permutation module (\ie $M=\Z[X]$ for some finite $G$-set $X$), is constructed as follows. Assuming for simplicity that $M$ has no $\Z$-torsion, by Artin's induction theorem, one can always find a positive integer $m$, $G$-permutation modules $A_1$ and $A_2$ and an injective $\Z[G]$ homomorphism $f:M^m\oplus A_1 \to A_2$ with finite cokernel. Then, identifying $A_2$ with $\dual{A}_2$, we can see $\dual{f}\circ f$ as a homomorphism $M^m\oplus A_1 \to \dual{\left(M^m\oplus A_1\right)}$ with finite cokernel $V$. In particular $V$ is trivial if $f$ is an isomorphism (this is the case in particular if $M$ is a permutation module). 

However one can have $\rc{\Theta}{V}=1$ regardless of whether there exists an isomorphism $f$ as above. We show in Section \ref{sec:dihedral} that in the dihedral case one has $\rc{\Theta_{D}}{W}=1$ for every finite $D$-module with $W\cong \pdual{W}$ and in fact
\begin{equation}\label{eq:mainintro}
\rc{\Theta_{D}}{M} = \left(\hht{0}{\Theta_{D}}{M}\hht{1}{\Theta_{D}}{M}\right)^{-1}
\end{equation}
for \textit{any} $\Z[D]$-module $M$ (possibly with $\Z$-torsion). This direct relation between regulator constants and cohomology can be used to derive explicit bounds on the possible values of the regulator constants, in terms of the $\Z$-ranks of $M^{D}$ and $M^{H}$ for $H\leq \Rho$ together with the order of $\tors{M^D}/q$ and $\tors{M^\Rho}/q$ (see Theorem \ref{th:bounds}).

In the last section of the paper we apply the main results of the previous sections to various arithmetic contexts. As above, let $L/k$ be a $G$-Galois extension of number fields and let $\OO{L}{}$ be the ring of integers of $L$. Yokoi proved that, when $G$ is cyclic, $\hht{i}{G}{\OO{L}{}}$ does not depend on $i\in\Z$ and conjectured that this was also the case when $G$ is not cyclic (see \cite{Yokoi1964}, \cite{Rosen1966}). The conjecture was disproved by Lee and Madan (see \cite{LeeMadan1969}): they showed that $\hht{1}{G}{\OO{L}{}}=1$ and $\hht{0}{G}{\OO{L}{}}\ne 1$ when $k=\Q$ and $L=\Q(\zeta_3, \sqrt[3]{2})$ with $\zeta_3$ a primitive third root of unity (so $G=D_{2q}$ with $q=3$). Our results can be used to shed more light on this question. In fact, by a classical result of Fr\"ohlich, $\OO{N}{}$ is factor equivalent to a free $\Z[G]$-module. This can be shown to be equivalent to 
\[
\rc{\Theta}{\OO{L}{}} = 1 
\] 
for any $G$-relation $\Theta$. Thus, when $G=D$, by \eqref{eq:mainintro} we obtain a relation between the cohomology groups of $\OO{L}{}$ which turns out to be enough to show that
\[
\hht{i}{D}{\OO{L}{}} = \hht{j}{D}{\OO{L}{}} \quad \text{if $i+j\equiv -1 \pmod 4$}
\]
(recall that the cohomology of $D$ is periodic of period $4$).

The regulator constant $\rc{\Theta}{\mathcal{O}_L^\times}$ is more interesting, since it is related to the special values of Dedekind zeta functions at $0$. In fact one can show that the following equality holds (see \cite[(2.1)]{BarteldeSmit2013})
\[
\rc{\Theta}{\mathcal{O}_L^\times}\rc{\Theta}{\Z}^{-1}\rc{\Theta}{\Z[S_\infty]} = \prod_{H\leq D} \left(\frac{\card{\roots{L^H}}}{\cl{L^H}{}}\right)^{2n_H}
\]
where $\roots{L^H}$ and $\cl{L^H}{}$ denote the group of roots of unity and the class number of $L^H$, respectively, and $S_\infty$ is the set of archimedean places of $L$. Restricting to the case $G=D$ and applying \eqref{eq:mainintro} to $M=\mathcal{O}_L^\times$, we obtain another proof of the main results of \cite{CaputoNuccio2020}, namely an equality 
\[ 
\frac{\cl{L}{}\cl{k}{2}}{\cl{F}{}\cl{K}{2}} = \frac{\hht{0}{D}{\mathcal{O}_L^\times}}{\hht{-1}{D}{\mathcal{O}_L^\times}}\frac{\hht{-1}{\Sigma}{\mathcal{O}_L^\times}}{\hht{0}{\Sigma}{\mathcal{O}_L^\times}},
\]
where $F=L^\Rho$ and $K=L^\Sigma$, and explicit bounds on the possible values of the class numbers ratio. We will also briefly sketch a proof of a similar result for algebraic $K$-groups of $\OO{L}{}$, which are related to special values of the Dedekind zeta function at negative integers.    

Another analogous situation where the above idea can be applied is the case of an elliptic curve $E$ defined over $k$. Assuming the finiteness of the Tate-Shafarevich group $\ts{L^H}$ for $H\leq G$, one has
\[
\rc{\Theta}{\rp{L}} = \prod_{H\leq D}\left(\frac{\card{\tors{\rp{L^H}}}^2}{\tn{L^H}\card{\ts{L^H}}}\right)^{n_H}
\]	
where $\tn{L^H}$ denotes the product of suitably normalised Tamagawa numbers over all finite places of $L^H$  and $\tors{\rp{L^H}}$ is the torsion subgroup of the $L^H$-rational points of $E$ (see \cite[Section 2.4]{BarteldeSmit2013}). Again restricting to the case $G=D$ and applying \eqref{eq:mainintro} to $M=\rp{L}$, we get 
\begin{equation}\label{eq:ecfintro}
\prod_{H\leq D_{2q}}\left(\frac{\card{\tors{\rp{L^H}}}^2}{\tn{L^H}\card{\ts{L^H}}}\right)^{n_H}= \left(\hht{0}{\Theta}{\rp{L}}\hht{1}{\Theta}{\rp{L}}\right)^{-1}.
\end{equation}
One also obtains information on the possible values of the left-hand side of the above equation in terms of the $\Z$-ranks $\rpr{L^H}$ of $\rp{L^H}$. This is particularly interesting since, contrarily to the case of units and $K$-groups described above, these ranks are hard to compute and there is no known explicit formula relating them to simpler invariants. For instance we use \eqref{eq:ecfintro} to prove that 
\[
\frac{q \rpr{F} -\rpr{L}}{q-1} \equiv v_q(\tn{F})- v_q(\tn{L}) \pmod 2.
\]
This was proved by T. and V. Dokchitser in the particular case where $q$ is a prime (see \cite[Proposition 4.17]{DokchitserDokchitser2010}).

\section{Regulator constants of dual modules}\label{sec:main}
In this section, $G$ is a finite group and $\Theta = \sum n_H H $ is a $G$-Brauer relation as in the Introduction. We introduce some notation that will be used throughout the paper. For a finitely generated $\Z[G]$-module $M$, we denote as usual by $M^G$ the submodule of elements of $M$ fixed by $G$ and by $M_G$ the quotient of $M$ by $I_GM$ where $I_G \subset \Z[G]$ is the augmentation ideal. We consider the $\Z$-dual 
\[
\dual{M} = \Hom(M,\Z)
\]
endowed with the standard $G$-action $gf(m) = f(g^{-1}m)$. If $f:M\to N$ is a homomorphism of $\Z[G]$-modules, we denote by $\dual{f}:\dual{N}\to \dual{M}$ the dual homomorphism. Similarly, for a finite $\Z[G]$-module $M$, we consider the $\Q/\Z$-dual
\[
\pdual{M} = \Hom(M,\Q/\Z)
\]
endowed with the same $G$-action as above. Finally, $\card{S}$ denotes the cardinality of a finite set $S$.

The main goal of this section is establishing a relation between regulator constants and cohomology of finitely generated $\Z[G]$-modules. Regulator constants were introduced by T. and V. Dokchitser in \cite{DokchitserDokchitser2010}, although a similar invariant was already studied by Kani in \cite{Kani1994}. Here we will be using a slightly modified definition which simplifies statements, especially in the case where $M$ has torsion (see \cite{Caputo2025}). For an abelian group $A$, let $\tors{A}$ and $\mt{A}=A/\tors{A}$ denote the torsion subgroup and the torsion-free quotient of $A$, respectively. 

\begin{definition}\label{def:rc}
Let $M$ be a finitely generated $\Z[G]$-module and let $\rcfieldext/\Q$ be a field extension. Let $\langle\cdot,\cdot\rangle~:~M~\times~M~\to~\rcfieldext$ be a $\Z$-linear $G$-invariant pairing that is non-degenerate on $\mt{M}$. Then the regulator constant $\rc{\Theta}{M}$ of $M$ with respect to the $G$-Brauer relation $\Theta=\sum_{H\leq G}n_H H$ is
\[
\rc{\Theta}{M} = \prod_{H\leq G}\left( \frac{1}{\card{\tors{M}^H}^2}\det\left(\frac{1}{\card{H}}\left.\langle\cdot,\cdot\rangle \right|_{\mt{M^H}} \right)\right)^{n_H} \in \Q^\times,
\]   
where the determinant appearing in the factor corresponding to the subgroup $H$ is evaluated on any $\Z$-basis of $\mt{M^H}$.
\end{definition}

Regulator constants do not depend on the bilinear pairing and are positive (since one can always choose a positive definite pairing). We now recall a useful characterisation of regulator constants, due to Bartel-de Smit (a similar result was also obtained previously by Kani). To state it in a more compact way, we will use the $q$-index introduced by Tate. For a homomorphism $f:A\to B$ of abelian groups with finite kernel and cokernel, the $q$-index of $f$ is defined as  
\[
\qindex{f} = \frac{\card{\Coker\, f}}{\card{\Ker\, f}}
\]
(see \cite[Chapitre II, \S 6.2]{Tate1984}). The $q$-index enjoys the following elementary properties (for a proof, see \cite[Propositions 3.4 and 3.6]{Kani1994}):
\begin{align}
\qindex{f} &= \qindex{\tors{f}}\qindex{\mt{f}} \label{eq:qtorssplit}\\
\qindex{\dual{f}} &= \qindex{f}\qindex{\tors{f}}^{-1} \label{eq:qdualsplit}.
\end{align}
Here $\tors{f}:\tors{A}\to\tors{B}$ and $\mt{f}:\mt{A}\to\mt{B}$ denote the homomorphism induced by $f$ on $\tors{A}$ and $\mt{A}=A/\tors{A}$ of $A$, respectively. 

We set 
\[
P_1 = \bigoplus_{\{H\leq G: n_H > 0\}} \Z[G/H]^{n_H}\quad \text{and}\quad P_2 = \bigoplus_{\{H\leq G: n_H < 0\}} \Z[G/H]^{-n_H},
\] 
so that $P_1\otimes \Q$ is $\Q[G]$-isomorphic to $P_2\otimes \Q$ since $\Theta=\sum n_H H $ is a Brauer relation. In particular there exists an injective homomorphism $\phi:P_1\to P_2$ with finite cokernel. Then the dual homomorphism $\dual{\phi}:\dual{P_2}\to \dual{P_1}$ also has finite cokernel. 

\begin{proposition}[Bartel-de Smit]\label{prop:rcalt} 
Let $\phi:P_1\to P_2$ be any injective homomorphism with finite cokernel. Then, for any finitely generated $\Z[G]$-module $M$, we have
\[
\rc{\Theta}{M} = \frac{\qindex{\left(\phi\otimes \mathrm{id}_M\right)^G}}{\qindex{\left(\dual{\phi}\otimes \mathrm{id}_M\right)^G}}.
\]
\end{proposition}
\begin{proof}
This is a rephrasing of \cite[Lemma 3.1]{BarteldeSmit2013}, using the canonical isomorphism 
\[
\Hom(P, M) \cong  \dual{P} \otimes_\Z M
\] 
(see \cite[II \S4, Corollaire de la Proposition 2]{BourbakiA1-3}). 
\end{proof}

We now prove a first relation between the regulator constant of a module and that of its dual. We first need a preliminary lemma. Recall from the introduction that, for $i\in\Z$ and $H\leq G$, $\HHT{i}{H}{M}$ is the $i$-th Tate cohomology group of $H$ with values in the finitely generated $\Z[G]$-module $M$ and $\hht{i}{H}{M}=\card{\HHT{i}{H}{M}}$.

\begin{lemma}\label{lemma:torscoinv-1}
Let $M$ be a finitely generated $\Z[G]$-module with no $\Z$-torsion. Then, for any subgroup $H\leq G$, 
\[
\tors{(M_H)} = \HHT{-1}{H}{M}.
\]
\end{lemma}
\begin{proof}
By definition of $\HHT{-1}{H}{M}$, there is an exact sequence
\[
0 \to \HHT{-1}{H}{M} \to M_H \stackrel{N_H}{\longrightarrow} M^H
\]
where $N_H=\sum_{h\in H}h\in\Z[H]$ is the norm. Since $M^H$ has no $\Z$-torsion, $\tors{(M_H)} \subseteq \HHT{-1}{H}{M}$ and equality holds since $\HHT{-1}{H}{M}$ is torsion.  
\end{proof} 

To simplify notation, for $i\in \mathbb{Z}$, we set 
\[
\hht{i}{\Theta}{M}=\prod_{H\leq G}\hht{i}{H}{M}^{n_H}.
\]
If $f:M\to N$ is a homomorphism of finitely generated $\Z[G]$-modules, we set
\[
\kdt{i}{\Theta}{f} = \prod_{H\leq G}\card{\Ker\,\HHT{i}{H}{f}}^{n_H}
\]
where 
\[
\HHT{i}{H}{f}:\HHT{i}{H}{M}\to\HHT{i}{H}{N}
\] 
is the map induced by $f$ in cohomology. Similarly, we denote by $f^H:M^H\to N^H$ and $f_H:M_H \to N_H$ the maps induced on invariants and coinvariants, respectively.

\begin{proposition}\label{prop:rcdual1st}
Let $M$ be a finitely generated $\Z[G]$-modules with no $\Z$-torsion. Then we have
\[
\rc{\Theta}{M}\rc{\Theta}{\dual{M}} = \hht{0}{\Theta}{M}^{-2}.
\]
\end{proposition} 
\begin{proof}
Observe that
\begin{equation}\label{eq:dualitycore}
\begin{aligned}
\qindex{\left(\phi\otimes \mathrm{id}_{\dual{M}}\right)_G}& = \qindex{\tors{\left(\left(\phi\otimes \mathrm{id}_{\dual{M}}\right)_G\right)}}\qindex{\mt{\left(\phi\otimes \mathrm{id}_{\dual{M}}\right)_G}}  \quad \text{(by \eqref{eq:qtorssplit})}\\
& = \qindex{\tors{\left(\left(\phi\otimes \mathrm{id}_{\dual{M}}\right)_G\right)}}\qindex{\dual{\left(\mt{\left(\phi\otimes \mathrm{id}_{\dual{M}}\right)_G}\right)}} \quad \text{(by \eqref{eq:qdualsplit})}\\
& = \qindex{\tors{\left(\left(\phi\otimes \mathrm{id}_{\dual{M}}\right)_G\right)}}\qindex{\dual{\left(\left(\phi\otimes \mathrm{id}_{\dual{M}}\right)_G\right)}}\\
& = \qindex{\tors{\left(\left(\phi\otimes \mathrm{id}_{\dual{M}}\right)_G\right)}}\qindex{\dual{\left(\left(\phi\otimes \mathrm{id}_{\dual{M}}\right)}\right)^G}\\
& = \qindex{\tors{\left(\left(\phi\otimes \mathrm{id}_{\dual{M}}\right)_G\right)}}\qindex{\left(\dual{\phi}\otimes \mathrm{id}_M\right)^G}.
\end{aligned}
\end{equation}  
The last equality follows from the existence of canonical isomorphisms $\dual{(P \otimes M)} \cong \dual{P} \otimes \dual{M}$, for a $\Z[G]$-module $P$ with no $\Z$-torsion (see \cite[II \S4, Corollaire 1 de la Proposition 4]{BourbakiA1-3}) and
$\ddual{M} \cong M$ (see \cite[II \S2, Proposition 13]{BourbakiA1-3}).

Moreover we have
\begin{equation}\label{eq:torscoinvh-1}
\qindex{\tors{\left(\left(\phi\otimes \mathrm{id}_{\dual{M}}\right)_G\right)}} = \hht{-1}{\Theta}{\dual{M}}^{-1}. 
\end{equation}
In fact, since both $\tors{((P_1\otimes_\Z M)_G)}$ and $\tors{((P_2\otimes_\Z M)_G)}$ are finite, we have
\[
\qindex{\tors{\left(\left(\phi\otimes \mathrm{id}_M\right)_G\right)}} =\frac{ \card{\tors{((P_2\otimes_\Z M)_G)}}}{\card{\tors{((P_1\otimes_\Z M)_G)}}} = \prod_{H\leq G} \card{\tors{(M_H)}}^{-n_H}
\]
by Shapiro's lemma. Then \eqref{eq:torscoinvh-1} follows from Lemma \ref{lemma:torscoinv-1}. Plugging \eqref{eq:torscoinvh-1} in \eqref{eq:dualitycore}, we deduce that
\begin{equation}\label{eq:invcoinvflip}
\qindex{\left(\dual{\phi}\otimes \mathrm{id}_M\right)^G} = \qindex{\left(\phi\otimes \mathrm{id}_{\dual{M}}\right)_G}\hht{-1}{\Theta}{\dual{M}}. 
\end{equation}
In a similar way we obtain
\begin{equation}\label{eq:invcoinvflip2}
\qindex{\left(\dual{\phi}\otimes \mathrm{id}_{\dual{M}}\right)^G} = \qindex{\left(\phi\otimes \mathrm{id}_{M}\right)_G}\hht{-1}{\Theta}{M}. 
\end{equation}
 
Observe also that  
\begin{equation}\label{eq:qvsherbrandequivariant}
\frac{\qindex{\left(\phi\otimes \mathrm{id}_M\right)^G}}{\qindex{\left(\phi\otimes \mathrm{id}_{M}\right)_G}} = \frac{\hht{-1}{\Theta}{M}}{\hht{0}{\Theta}{M}}.
\end{equation}
To see this, consider the commutative diagram with exact rows
\[
\begin{tikzcd}
0 \arrow[r]& \displaystyle{\bigoplus_{\substack{H\leq G, \\n_H \geq 0}}} \HHT{-1}{H}{M}^{n_H} \arrow[r]\arrow[d]& \left(P_1\otimes_\Z M\right)_G \arrow[r,"N_1"]\arrow[d,"(\phi\otimes\mathrm{id}_M)_G"]& \left(P_1\otimes_\Z M\right)^G\arrow[r]\arrow[d,"(\phi\otimes\mathrm{id}_M)^G"]& \displaystyle{\bigoplus_{\substack{H\leq G, \\n_H \geq 0}}} \HHT{0}{H}{M}^{n_H}\arrow[r]& 0\\
0 \arrow[r]& \displaystyle{\bigoplus_{\substack{H\leq G, \\n_H \leq 0}}} \HHT{-1}{H}{M}^{-n_H} \arrow[r]& \left(P_2\otimes_\Z M\right)_G \arrow[r,"N_2"]& \left(P_2\otimes_\Z M\right)^G\arrow[r]& \displaystyle{\bigoplus_{\substack{H\leq G, \\n_H \leq 0}}} \HHT{0}{H}{M}^{-n_H}\arrow[r]& 0
\end{tikzcd}
\]
Here $N_1=\bigoplus_{H\leq G, n_H\geq 0} N_H^{n_H}$ is the homomorphism induced by the norms $N_H$ on 
\[
\left(P_1\otimes _\Z M\right)_G = \left(\displaystyle{\bigoplus_{\substack{H\leq G, \\n_H \geq 0}}} \Z[G/H]^{n_H} \otimes M\right)_G = \displaystyle{\bigoplus_{\substack{H\leq G, \\n_H \geq 0}}} (M_H)^{n_H}
\]
and similarly for $N_2$. Splitting the above diagram in two diagrams, each having 3-term short exact sequences as rows, and applying the snake lemma to each of them, we deduce \eqref{eq:qvsherbrandequivariant}.  
Therefore
\begin{align*}
\rc{\Theta}{M}\rc{\Theta}{\dual{M}} &= \frac{\qindex{\left(\phi\otimes \mathrm{id}_M\right)^G}}{\qindex{\left(\dual{\phi}\otimes \mathrm{id}_M\right)^G}}
\frac{\qindex{\left(\phi\otimes \mathrm{id}_{\dual{M}}\right)^G}}{\qindex{\left(\dual{\phi}\otimes \mathrm{id}_{\dual{M}}\right)^G}}\quad \text{by Prop. \ref{prop:rcalt}}\\
&=\frac{\qindex{\left(\phi\otimes \mathrm{id}_M\right)^G}}{\qindex{\left(\phi\otimes \mathrm{id}_{\dual{M}}\right)_G}}\hht{-1}{\Theta}{\dual{M}}^{-1}
\frac{\qindex{\left(\phi\otimes \mathrm{id}_{\dual{M}}\right)^G}}{\qindex{\left(\phi\otimes \mathrm{id}_{M}\right)_G}}\hht{-1}{\Theta}{M}^{-1} \quad \text{by  \eqref{eq:invcoinvflip} and \eqref{eq:invcoinvflip2}}\\
&=\frac{\qindex{\left(\phi\otimes \mathrm{id}_M\right)^G}}{\qindex{\left(\phi\otimes \mathrm{id}_{M}\right)_G}}\hht{-1}{\Theta}{M}^{-1}
\frac{\qindex{\left(\phi\otimes \mathrm{id}_{\dual{M}}\right)^G}}{\qindex{\left(\phi\otimes \mathrm{id}_{\dual{M}}\right)_G}}\hht{-1}{\Theta}{\dual{M}}^{-1}\\
&=\hht{0}{\Theta}{M}^{-1}\hht{0}{\Theta}{\dual{M}}^{-1} \quad \text{by  \eqref{eq:qvsherbrandequivariant}.}\\
&=\hht{0}{\Theta}{M}^{-2}
\end{align*}
where the last equality follows from duality (see \cite[Proposition 3.1.2]{NSW2013}). 
\end{proof}

We now establish a second relation between $\rc{\Theta}{M}$ and $\rc{\Theta}{\dual{M}}$. 
Let $\repring{\Q}{G}$ denote the rational representation ring of $G$ and let $\burnside{\Q}{G}\subseteq \repring{\Q}{G}$ denote the subring generated by rational permutation representations. By a version of Artin’s induction
theorem (see \cite[Theorem 30]{Serrelrfg}), $\repring{\Q}{G}/\burnside{\Q}{G}$ has finite exponent, which we denote by $\bexp{G}$. In particular, for a finitely generated $\Z[G]$-module $M$ with no $\Z$-torsion there always exist permutation $\Z[G]$-modules $A_1, A_2$, a positive integer $m\mid \bexp{G}$ and an injective $\Z[G]$-homomorphism $f: M^m \oplus A_1 \to A_2$ with finite cokernel. The exponent $\bexp{G}$ has been computed for several classes of groups: for instance, it is known that $\bexp{G}=1$ for $p$-groups, symmetric groups and some metacyclic groups, including dihedral groups, but in general $\bexp{G}$ is nontrivial (see \cite{Rasmussen1977} and references therein).


\begin{proposition}\label{prop:rcdual2nd}
Let $M$ be a finitely generated $\Z[G]$-module with no $\Z$-torsion. Let $m\mid \bexp{G}$ be a positive integer such that there exist permutation $\Z[G]$-modules $A_1, A_2$ and an injective $\Z[G]$-homomorphism $f: M^m \oplus A_1 \to A_2$ with finite cokernel (such $m$ always exists by Artin's induction theorem). Then we have
\[
\rc{\Theta}{\dual{M}}^m =  \rc{\Theta}{V}\rc{\Theta}{M}^m\hht{1}{\Theta}{M}^{2m}
\]
for some finite $\Z[G]$-module $V$ such that $V\cong \pdual{V}$ as $\Z[G]$-modules.
\end{proposition} 
\begin{proof}
Recall that a permutation module $\Z[S]$ on a $G$-set $S$ has a perfect symmetric $\Z$-linear $G$-invariant pairing 
\[
\langle\,,\rangle:\Z[S]\times \Z[S] \to \Z
\] 
defined by
\[
\langle\sum_{s\in S}x_ss,\sum_{s\in S}y_ss\rangle=\sum_{s\in S}\sum_{t\in S}x_sy_t\delta_{s,t}
\]
where $\delta_{s,t} = 1$ if $s=t$ and $0$ otherwise. In particular we have a $\Z[G]$-isomorphism
\begin{equation}\label{eq:curry}
\varphi_{\Z[S]}:\Z[S] \to \dual{\Z[S]},\quad \sum_{s\in S}a_s s \mapsto\langle \sum_{s\in S}a_s s, \cdot\rangle.
\end{equation}
Now $F=\dual{f}\circ\varphi_{A_2}\circ f:M^m\oplus A_1 \to \dual{(M^m\oplus A_1)}$ is injective and has finite cokernel 
\[
V=\Coker(F).
\] 
By \cite[Proposition 2.12]{Caputo2025}, we have
\[
\rc{\Theta}{\dual{(M^m\oplus A_1)}} = \rc{\Theta}{M^m\oplus A_1} \rc{\Theta}{V} \kdt{1}{\Theta}{F}^2.
\]
Observe now that, for $H\leq G$, $\HHT{1}{H}{A_i}=0$ for $i=1,2$: permutation modules are direct sums of modules of the form $\Z[G/U]$ for $U\leq G$ and $\Z[G/U]$ is $\Z[H]$-isomorphic to a direct sum of modules of the form $\Z[H/H']$ for $H'\leq H$. Since $\HHT{1}{H}{\Z[H/H']}=\HHT{1}{H'}{\Z}=0$ (by Shapiro's lemma), we obtain $\HHT{1}{H}{A_i}=0$.
Now, for any $H\leq G$, $\HHT{1}{H}{F}$ factors through $\HHT{1}{H}{f}$ which is $0$ since $\HHT{1}{H}{A_2}=0$. In particular
\[
\kdt{1}{\Theta}{F} = \hht{1}{\Theta}{M^m\oplus A_1} = \hht{1}{\Theta}{M^m},
\] 
since $\HHT{1}{H}{A_1}=0$. Since regulator constants are multiplicative with respect to direct sums and $A_1$ is self-dual (see \eqref{eq:curry}), we get 
\[
\rc{\Theta}{\dual{M}}^m = \rc{\Theta}{M}^m \rc{\Theta}{V}\hht{1}{\Theta}{M}^{2m},
\]  
showing the first assertion.

To prove that $V\cong \pdual{V}$ as $\Z[G]$-modules we argue as follows. For a finitely generated $\Z[G]$-module $N$ with no $\Z$-torsion, let 
$\Phi_N:N\to \ddual{N}$ denote the canonical $\Z[G]$-isomorphism defined by
\[
\Phi_N(n): \psi \mapsto \psi(n), \quad \psi \in \dual{N}.
\]
Set $M_1=M^m \oplus A_1$: we claim that the following diagram is commutative
\begin{equation}\label{eq:symmetrydiag}
\begin{tikzcd}
M_1 \arrow[r, "f"]\arrow[d, "\Phi_{M_1}"]& A_2 \arrow[r,"\varphi_{A_2}"]\arrow[d,"\Phi_{A_2}"]& \dual{A_2} \arrow[r,"\dual{f}"] \dar[equal]& \dual{M_1}\dar[equal]\\
\ddual{M_1} \arrow[r,"\ddual{f}"]& \ddual{A_2} \arrow[r,"\dual{\varphi_{A_2}}"]& \dual{A_2} \arrow[r,"\dual{f}"]& \dual{M_1}
\end{tikzcd}
\end{equation}
The rightmost square is clearly commutative and the same holds for the leftmost one, since $\Phi$ is functorial. As for the central square, observe that, for $x, y\in A_2$,
\[
\dual{\varphi_{A_2}}(\Phi_{A_2}(x))(y) = \Phi_{A_2}(x)(\varphi_{A_2}(y)) = \varphi_{A_2}(y)(x) = \langle y,x\rangle = \langle x,y\rangle = \varphi_{A_2}(x)(y).
\]
This means that the central square of \eqref{eq:symmetrydiag} is commutative and the claim is proved. Now consider the diagram
\begin{equation}
\begin{tikzcd}
0\arrow[r]&M_1 \arrow[r, "F"]\arrow[d, "\Phi_{M_1}"]& \dual{M_1} \arrow[r]\dar[equal]& V \arrow[r] \arrow[d]& 0\\
0\arrow[r]&\ddual{M_1} \arrow[r,"\dual{F}"]& \dual{M_1} \arrow[r]& \Ext{1}{\Z}{V}{\Z} \arrow[r]& 0
\end{tikzcd}
\end{equation}
The top row is exact by definition of $V$. The bottom row is also exact since $\dual{M_1}$ has no $\Z$-torsion (so $\Ext{1}{\Z}{\dual{M_1}}{\Z}=0$) and $V$ is torsion (so $\Ker\,\dual{F}=\dual{V}=0$). Moreover we just proved that the leftmost square is commutative, so that the vertical map $V\to\Ext{1}{\Z}{V}{\Z}$ can be defined as the one making the rightmost square commutative. Finally, since the leftmost and central vertical arrows are bijective, so is the rightmost one, inducing an isomorphism $V\cong \pdual{V}$ since
\[
\Ext{1}{\Z}{V}{\Z} \cong \pdual{V}
\] 
as $\Z[G]$-modules (apply the left-exact functor $\Hom_\Z(V,\cdot)$ to $0\to \Z\to\Q\to\Q/\Z\to 0$ and use $\Hom_\Z(V,\Q) = \Ext{1}{\Z}{V}{\Q} = 0$).  
\end{proof}

%

Combining Propositions \ref{prop:rcdual1st} and \ref{prop:rcdual2nd} we get the following result.

\begin{corollary}\label{cor:rcformulatf} 
Let $M$ be a finitely generated $\Z[G]$-module with no $\Z$-torsion and let $m$ be as in Proposition \ref{prop:rcdual2nd}. Then we have
\[
\rc{\Theta}{M}^{2m} = \rc{\Theta}{V}^{-1} \left(\hht{0}{\Theta}{M}\hht{1}{\Theta}{M}\right)^{-2m}
\]
for some finite $\Z[G]$-module $V$ such that $V\cong \pdual{V}$ as $\Z[G]$-modules.
\end{corollary}

We now turn to the case of a general finitely generated $\Z[G]$-module, not necessarily $\Z$-torsion free. 

\begin{theorem}\label{th:rcformula}
Let $M$ be a finitely generated $\Z[G]$-module and let $m$ be as in Proposition \ref{prop:rcdual2nd} for $\mt{M}$. Then we have
\[
\rc{\Theta}{M}^{2m} = \rc{\Theta}{\tors{M}}^{2m} \rc{\Theta}{V}^{-1} \left(\frac{\hht{0}{\Theta}{\tors{M}}\hht{1}{\Theta}{\tors{M}}}{\hht{0}{\Theta}{M}\hht{1}{\Theta}{M}}\right)^{2m}\kdt{0}{\Theta}{\iota}^{-2m}\kdt{2}{\Theta}{\iota}^{-2m}
\]
for some finite $\Z[G]$-module $V$ such that $V\cong \pdual{V}$ as $\Z[G]$-modules.
\end{theorem}
\begin{proof}
Consider the exact sequence
\begin{equation}\label{eq:torsiones}
0 \rightarrow \tors{M} \stackrel{\iota}{\longrightarrow} M \longrightarrow \mt{M}\rightarrow 0.
\end{equation}
On the one hand, by \cite[Proposition 2.12]{Caputo2025} we have
\begin{equation}\label{eq:rctf1}
\rc{\Theta}{\mt{M}} = \rc{\Theta}{M}\rc{\Theta}{\tors{M}}^{-1}\kdt{1}{\Theta}{\iota}^{-2}.
\end{equation}
On the other hand, by Corollary \ref{cor:rcformulatf}, we have
\begin{equation}\label{eq:rctf2}
\rc{\Theta}{\mt{M}}^{2m} = \rc{\Theta}{V}^{-1} \left(\hht{0}{\Theta}{\mt{M}}\hht{1}{\Theta}{\mt{M}}\right)^{-2m}
\end{equation}
for some finite $\Z[G]$-module $V$ with $V\cong\pdual{V}$. Moreover, from the exact cohomology sequence induced by \eqref{eq:torsiones}, we obtain, for every $H\leq G$,
\begin{align*}
\hht{0}{H}{\mt{M}}&= \frac{\hht{0}{H}{M}}{\hht{0}{H}{\tors{M}}}\card{\Ker\,\HHT{0}{H}{\iota}}\card{\Ker\,\HHT{1}{H}{\iota}},\\
\hht{1}{H}{\mt{M}} &=\frac{\hht{1}{H}{M}}{\hht{1}{H}{\tors{M}}}\card{\Ker\,\HHT{1}{H}{\iota}}\card{\Ker\,\HHT{2}{H}{\iota}}.
\end{align*}
We deduce that
\[
\hht{0}{\Theta}{\mt{M}}\hht{1}{\Theta}{\mt{M}} = \frac{\hht{0}{\Theta}{M}\hht{1}{\Theta}{M}}{\hht{0}{\Theta}{\tors{M}}\hht{1}{\Theta}{\tors{M}}}\kdt{0}{\Theta}{\iota}\kdt{1}{\Theta}{\iota}^2\kdt{2}{\Theta}{\iota}.
\]
We can then rewrite \eqref{eq:rctf2} as
\[
\rc{\Theta}{\mt{M}}^{2m} = \rc{\Theta}{V}^{-1} \left(\frac{\hht{0}{\Theta}{\tors{M}}\hht{1}{\Theta}{\tors{M}}}{\hht{0}{\Theta}{M}\hht{1}{\Theta}{M}}\right)^{2m}\kdt{0}{\Theta}{\iota}^{-2m}\kdt{1}{\Theta}{\iota}^{-4m}\kdt{2}{\Theta}{\iota}^{-2m}.
\]
Plugging this in \eqref{eq:rctf1}, we get the statement.
\end{proof}

Although we won't need it in the rest of the paper, we prove the following result on regulator constants and cohomology of finite modules.

\begin{proposition}
Let $M$ be a finite $\Z[G]$-module. Then
\[
\frac{\rc{\Theta}{M}}{\rc{\Theta}{\pdual{M}}} = \left(\frac{\hht{-1}{\Theta}{M}}{\hht{0}{\Theta}{M}}\right)^2.
\]
\end{proposition}
\begin{proof}
For a subgroup $H\leq G$, observe that $\left(\pdual{M}\right)^H=\pdual{\left(M_H\right)}$ which is isomorphic to $M_H$ as abstract abelian groups. Moreover the norm $N_H=\sum_{h\in H} h$ induces an exact sequence
\[
0 \to \HHT{-1}{H}{M}\to M_H\to M^H \to \HHT{0}{H}{M}\to 0.
\]
Therefore
\begin{align*}
\frac{\rc{\Theta}{M}}{\rc{\Theta}{\pdual{M}}} &= \prod_{H\leq G} \left(\frac{\card{\left(\pdual{M}\right)^H}}{\card{M^H}}\right)^{2n_H}\\
&= \prod_{H\leq G} \left(\frac{\card{M_H}}{\card{M^H}}\right)^{2n_H}\\
&= \left(\frac{\hht{-1}{\Theta}{M}}{\hht{0}{\Theta}{M}}\right)^2. \qedhere
\end{align*} 
\end{proof}

\section{The dihedral case}\label{sec:dihedral}
In this section we show that the results of the previous section can be improved in the case where the group $G$ is dihedral. Let $q>1$ be an odd integer and let $D=D_q$ denote the dihedral group of order $2q$:
\begin{equation}\label{eq:dgroup}
D=D_q = \langle \rho, \sigma : \rho^q = \sigma^2 = 1, \sigma\rho\sigma^{-1}=\rho^{-1}\rangle, \quad \Rho = \Rho_q = \langle \rho \rangle \triangleleft D, \quad \Sigma = \langle\sigma\rangle.
\end{equation}
The group $D$ has the following Brauer relation
\begin{equation}\label{eq:drelation}
\Theta = \Theta_D = 1 +2D -\Rho - 2\Sigma.
\end{equation}

If $M$ is a uniquely $2$-divisible $\Z[D]$-module, then one can consider the idempotent decomposition
\[
M = M^+ \oplus M^-
\]
where
\[
M^+ = M^\Sigma = \frac{1+\sigma}{2}M, \quad M^- = \frac{1-\sigma}{2}M.
\] 

\begin{proposition}\label{prop:finitemoduledihedral}
Let $M$ be a finite $\Z[D]$-module. Then
\begin{equation}\label{eq:drcfm} 
\frac{\card{M}\card{M^D}^2}{\card{M^\Rho}\card{M^\Sigma}^2} = \frac{\hht{0}{D}{M}}{\hht{-1}{D}{M}}.
\end{equation}
In particular
\[
\rc{\Theta}{M}= \left(\frac{\hht{-1}{D}{M}}{\hht{0}{D}{M}}\right)^{2} = \frac{\hht{-1}{\Theta}{M}}{\hht{0}{\Theta}{M}}.
\]
If, moreover, $M\cong \pdual{M}$ as $\Z[D]$-modules, then
\[
\frac{\card{M}\card{M^D}^2}{\card{M^\Rho}\card{M^\Sigma}^2} = \rc{\Theta}{M} = 1.
\]
\begin{proof}
We can analyse separately the $\ell$-primary components $M(\ell)$ of $M$, for $\ell$ prime. Assume first that $\ell \nmid q$. Then the left-hand side of \eqref{eq:drcfm} is $1$ for $M(\ell)$, by \cite[Lemma 2.6]{CaputoNuccio2020} (using that $M(\ell)^{\Sigma}\cong M(\ell)^{\Sigma'}$ where $\Sigma'=\langle\sigma\rho\rangle$, as explained in \cite[p. 6]{CaputoNuccio2020}). As for the right-hand side, we have $\HHT{i}{\Rho}{M(\ell)}=1$ (since $\HHT{i}{\Rho}{M(\ell)}$ has exponent dividing $\card{\Rho}$) and
\[
\HHT{i}{D}{M(\ell)} \cong \HHT{i}{\Sigma}{M(\ell)}
\]
for every $i\in\Z$: in fact, for $\ell \nmid q$ both cohomology groups are trivial and for $\ell=2$ they are isomorphic by \cite[Lemma 2.9]{CaputoNuccio2020}. The right-hand side of \eqref{eq:drcfm} is therefore $1$ for $M(\ell)$, since the $\Sigma$-Herbrand quotient of a finite module is $1$. This settles the case $\ell\nmid q$.  

Suppose now that $M$ is $\ell$-primary, with $\ell \mid q$: in particular, $M$ is uniquely $2$-divisible. Therefore we can rewrite the left-hand side of \eqref{eq:drcfm} as
\[
\frac{\card{M}\card{M^D}^2}{\card{M^\Rho}\card{M^\Sigma}^2} = \frac{\card{M/M^\Rho}}{\card{\left(M/M^\Rho\right)^+}^2} = \frac{\card{\left(M/M^\Rho\right)^-}}{\card{\left(M/M^\Rho\right)^+}}.
\]
As for the right-hand side of \eqref{eq:drcfm}, we have
\begin{align*}
\frac{\hht{0}{D}{M}}{\hht{-1}{D}{M}} &= \frac{\card{M^D}\card{I_DM}}{\card{N_DM}\card{M[N_D]}}\\ 
&= \frac{\card{M^D} }{\card{M/I_DM}}\\
&= \frac{\card{M^D} \card{I_DM/I_\Rho M}}{\card{M/I_\Rho M}}\\ 
&= \frac{\card{\left(M^\Rho\right)^+}}{\card{\left(M_\Rho\right)^+}}\\ 
&= \frac{\card{\left(I_\Rho M\right)^+}}{\card{\left(M/M^\Rho\right)^+}}.  
\end{align*}
Here we have used that
\[
\left(M/I_\Rho M\right)^- = \frac{1-\sigma}{2} \left(M/I_\Rho M\right) = \left((1-\sigma)M + I_\Rho M\right)/I_\Rho M = I_DM/I_\Rho M.
\]
We are then left to prove that
\[
\card{\left(M/M^\Rho\right)^-} = \card{\left(I_\Rho M\right)^+}.
\]
Consider the $\Z[\Rho]$-module isomorphism $f: M/M^\Rho \to I_\Rho M$ induced by multiplication by $(1-\rho)$. Let $x\in (M/M^\Rho)^-$ be represented by $m\in M$: in particular $\sigma m = -m + m'$ for some $m'\in M^\Rho$. Then $f(x)$ is fixed by $\sigma \rho^{-1}$: indeed,
\[
\sigma \rho ^{-1} \left((1-\rho) m\right) = \sigma (\rho ^{-1}-1) m = (\rho - 1)\sigma m = (\rho - 1)(-m  +m')= (1-\rho)m.
\]
Conversely, take $m''\in M$ and suppose that $(1-\rho)m''\in I_\Rho M$ is fixed by $\sigma\rho^{-1}$. In particular 
\[
(1-\rho)m'' = \sigma \rho ^{-1} \left((1-\rho) m''\right) = \sigma (\rho ^{-1}-1) m'' = -(1-\rho)\sigma m''.
\]
This implies that $(1+\sigma)m'' \in M^\Rho$ and therefore the class of $m''$ in $M/M^\Rho$ belongs to $(M/M^\Rho)^-$. Summarizing, we have shown that 
\[
(M/M^\Rho)^- \cong \left(I_\Rho M\right)^{\langle\sigma \rho^{-1}\rangle}
\]
as abelian groups. But, for any $D$-submodule $N\subseteq M$, multiplication by $\rho^t$ with $2t\equiv -1 \pmod q$ gives an isomorphism of abelian groups 
\[
N^{\langle\sigma \rho^{-1}\rangle} \cong N^\Sigma =
N^+
\]
since, if $n \in N^{\langle\sigma \rho^{-1}\rangle}$,
\[
\sigma(\rho^t n) = \sigma \rho^{t+1} \rho^{-1} n = \rho^{-t-1} \sigma \rho^{-1} n = \rho^{-t-1} n = \rho^t n.
\]
This concludes the proof of the first assertion of the proposition. The second assertion is just a reformulation of the first, using the fact that the Herbrand quotient of finite modules is trivial.

Finally, the last assertion follows by duality: for any $i\in\Z$ we have
\[
\HHT{i}{D}{\pdual{M}} \cong \pdual{\HHT{-1-i}{D}{M}}
\]
(see \cite[Proposition 3.1.1]{NSW2013}).
\end{proof}
\end{proposition}

\begin{lemma}\label{lemma:dcf} 
Let $f:M\to N$ be a homomorphism of finitely generated $\Z[D]$-modules. Then, for every $i \in \Z$, we have
\[
\kdt{i}{\Theta}{f}\kdt{i+2}{\Theta}{f} = 1 
\]
and, in particular,
\[
\hht{i}{\Theta}{M}\hht{i+2}{\Theta}{M} =1.
\]
\end{lemma}
\begin{proof}
We will prove both assertions at once, although the second also follows from the first by taking $M=N$ and $f=0$. Observe that, for any prime $\ell$, there is a functorial isomorphism
\begin{equation}\label{eq:zlcohom}
\HHT{i}{H}{M\otimes_\Z\Z_\ell} \cong \HHT{i}{H}{M}\otimes_\Z\Z_\ell.
\end{equation}
It is therefore enough to consider the case where $M$ and $N$ are $\Z_\ell[D]$-modules for some prime $\ell$. 

When $\ell = 2$, $\Ker\, \HHT{i}{\Rho}{f} = \HHT{i}{\Rho}{M}=0$ since $M$ is uniquely $q$-divisible. Thanks to \cite[Lemma 2.9]{CaputoNuccio2020}, restriction induces an isomorphism $\HHT{i}{D}{M} \cong \HHT{i}{\Sigma}{M}$. Since this isomorphism is functorial, we obtain an isomorphism $\Ker\, \HHT{i}{D}{f}\cong \Ker\, \HHT{i}{\Sigma}{f}$, so the case $\ell = 2$ is proved. 

When $\ell \ne 2$, $M$ and $N$ are uniquely $2$-divisible and, in particular, $\Ker\, \HHT{i}{\Sigma}{f} = \HHT{i}{\Sigma}{M} =0$. By \cite[Proposition 2.1]{CaputoNuccio2020}, restriction induces a functorial isomorphism $\HHT{i}{D}{M}\cong \HHT{i}{\Rho}{M}^+$ and the Tate periodicity isomorphism $\HHT{i}{\Rho}{M}\cong \HHT{i+2}{\Rho}{M}$ followed by restriction induces a functorial isomorphism $\HHT{i}{\Rho}{M}^- \cong \HHT{i+2}{D}{M}$. 
We therefore obtain isomorphisms
\[
\Ker\, \HHT{i}{\Rho}{f} \cong \Ker\, \HHT{i}{D}{f} \oplus \Ker\, \HHT{i+2}{D}{f},\quad \HHT{i}{\Rho}{M} \cong \HHT{i}{D}{M} \oplus \HHT{i+2}{D}{M} .
\]
We deduce that
\begin{align*}
\hht{i}{\Theta}{M}\hht{i+2}{\Theta}{M} &= \frac{\hht{i}{D}{M}^2}{\hht{i}{\Rho}{M}}\frac{\hht{i+2}{D}{M}^2}{\hht{i+2}{\Rho}{M}}\\
&=\frac{\hht{i}{D}{M}^2}{\hht{i}{D}{M}\hht{i+2}{D}{M}}\frac{\hht{i+2}{D}{M}^2}{\hht{i+2}{D}{M}\hht{i+4}{D}{M}}\\
&=\frac{\hht{i}{D}{M}}{\hht{i+4}{D}{M}}
\end{align*}
and similarly
\[
\kdt{i}{\Theta}{f}\kdt{i+2}{\Theta}{f} = \frac{\card{\Ker\,\HHT{i}{D}{f}}}{\card{\Ker\,\HHT{i+4}{D}{f}}}
\]
This concludes the proof, since $D$ has periodic cohomology with minimal period $4$. This means that $\HHT{4}{D}{\Z}$ is cyclic of order $4$ and cup product with a generator of $\HHT{4}{D}{\Z}$ gives a functorial isomorphism
\[
\HHT{i}{D}{M} \cong \HHT{i+4}{D}{M}
\] 
(see \cite[Chapter IX, Exercise 6]{Brown1994}).
\end{proof}

\begin{theorem}\label{th:dihedralrccohom}
Let $M$ be  finitely generated $\Z[D]$-module. Then
\begin{equation}\label{eq:dihedralrccohomom}
\rc{\Theta}{M} = \left(\hht{0}{\Theta}{M}\hht{1}{\Theta}{M}\right)^{-1} = \frac{\hht{-1}{\Theta}{M}}{\hht{0}{\Theta}{M}}.
\end{equation}
\end{theorem}
\begin{proof}
By Theorem \ref{th:rcformula}, there exist a positive integer $m$ and a finite $\Z[D]$-module $V$ with $V\cong \pdual{V}$ such that
\begin{equation}\label{eq:d0}
\rc{\Theta}{M}^{2m} = \rc{\Theta}{\tors{M}}^{2m} \rc{\Theta}{V}^{-1} \left(\frac{\hht{0}{\Theta}{\tors{M}}\hht{1}{\Theta}{\tors{M}}}{\hht{0}{\Theta}{M}\hht{1}{\Theta}{M}}\right)^{2m}\kdt{0}{\Theta}{\iota}^{-2m}\kdt{2}{\Theta}{\iota}^{-2m}
\end{equation} 
where $\iota:\tors{M}\to M$ is the inclusion. Proposition \ref{prop:finitemoduledihedral}, applied to $V$, gives
\begin{equation}\label{eq:d1}
\rc{\Theta}{V} = 1. 
\end{equation} 
Applying again Proposition \ref{prop:finitemoduledihedral} to $\tors{M}$ and using Lemma \ref{lemma:dcf} we get
\begin{equation}\label{eq:d2}
\rc{\Theta}{\tors{M}}^2\left(\hht{0}{\Theta}{\tors{M}}\hht{1}{\Theta}{\tors{M}}\right)^{2} = \rc{\Theta}{\tors{M}}^2\left(\frac{\hht{0}{\Theta}{\tors{M}}}{\hht{-1}{\Theta}{\tors{M}}}\right)^{2}=1.
\end{equation}
Using once more Lemma \ref{lemma:dcf} we also obtain
\begin{equation}\label{eq:d3}
\kdt{0}{\Theta}{\iota}^{-2}\kdt{2}{\Theta}{\iota}^{-2}=1.
\end{equation}
Plugging \eqref{eq:d1}, \eqref{eq:d2} and \eqref{eq:d3} in \eqref{eq:d0}, we get the first equality of the theorem. The second equality comes from Lemma \ref{lemma:dcf}.
\end{proof}

\begin{remark}\label{rmk:dihedralrcval}
Let $\ell$ be a prime. Theorem \ref{th:dihedralrccohom} implies that $v_\ell(\rc{\Theta}{M})=0$ if $\ell\nmid q$ (see \cite[Proposition 3.9]{Bartel2012}). It is clear that if $\ell\nmid 2q$, then the $\ell$-adic valuation of the right-hand side of \eqref{eq:dihedralrccohomom} is $0$. If $\ell=2$, the same holds by \cite[Lemma 2.9]{CaputoNuccio2020}.
\end{remark}

\begin{remark}
A result of Torzewski (see \cite[Theorem 6.8]{Torzewski2020}) adapted to our context implies that the genus\footnote{Recall that two $\Z$-torsion free $\Z[D]$-modules $M_1$ and $M_2$ belong to the same genus if and only if $M_1\otimes_\Z \Z_\ell\cong M_2\otimes_\Z \Z_\ell$ as $\Z_\ell[D]$-modules for any prime $\ell$. One can show that, if $M_1$ and $M_2$ are in the same genus, then $\rc{\Theta}{M_1}=\rc{\Theta}{M_2}$.} of a finitely generated $\Z$-torsion free $\Z[D]$-module $M$ is determined by the following data, for every prime $\ell$:
\begin{itemize}
	\item the $\Q_\ell[D]$-isomorphism class of $M\otimes_\Z \Q_\ell$;
	\item for $\ell\mid q$, the $\ell$-adic valuation $v_\ell(\rc{\Theta_H}{M})$ where $H\leq D$ runs over the subgroups of the form $D_{q(\ell)}$ where $q(\ell)$ is the exact power of $\ell$ dividing $q$;
	\item  for $\ell\mid 2q$, the Yakovlev diagram of $\Z[N_D(D(\ell))]$-modules
		\[
		\begin{tikzcd}
		\HHT{1}{D(\ell)}{M\otimes \Z_\ell}\arrow[r,shift left=2pt,"res"] & 
		\arrow[l,shift left=2pt,"cor"] \HHT{1}{D(\ell)^{\ell}}{M\otimes \Z_\ell}\arrow[r,shift left=2pt,"res"] &
		\arrow[l,shift left=2pt,"cor"] \cdots \arrow[r,shift left=2pt,"res"] & \arrow[l,shift left=2pt,"cor"] \HHT{1}{D(\ell)^{q(\ell)}}{M\otimes \Z_\ell}
		\end{tikzcd}
		\]
	where $D(\ell)$ is the $\ell$-Sylow subgroup of $D$ whose order is $q(\ell)$, $D(\ell)^a$ denotes the subgroup of $a$-th powers of $D(\ell)$ and $N_D(D(\ell))$ is the normaliser of $D(\ell)$ in $D$.
\end{itemize} 

In view of this, Theorem \ref{th:dihedralrccohom} implies that the genus of $M$ is fully determined by the $\Q_\ell[D]$-isomorphism class of $M\otimes_\Z \Q_\ell$ together with purely cohomological information.
\end{remark}

For a finitely generated $\Z$-module $A$, let $\rk{A}$ denote the rank of $A$  (\ie the dimension of $A\otimes_\Z \Q$ as a $\Q$-vector space). We shall now use the above corollary to obtain explicit bounds on the values of $\rc{\Theta}{M}$ in terms of $\rk{M^H}$ for $H\leq D$. Recall that the $\Rho$-Herbrand quotient of $M$ is defined as  
\[
\hq{\Rho}{M} = \frac{\hht{0}{\Rho}{M}}{\hht{-1}{\Rho}{M}}.
\]
By a classical result of Rosen (see also \cite[Theorem 10.3]{Chevalley1953}), the Herbrand quotient of $M$ can be expressed in terms of $\Z$-ranks of submodules of $M$ as follows. Recall that for a prime $\ell$, $v_\ell$ denotes the $\ell$-adic valuation.

\begin{lemma}[{\cite[Lemma]{Rosen1966}}]\label{lemma:rosen}
Let $M$ be a finitely generated $\Z[\Rho]$-module and let $q = \prod_\ell \ell^{t(\ell)}$ be the prime decomposition of $q$. Then
\[
v_\ell(\hq{\Rho}{M}) = t(\ell)r(q) - \sum_{i=1}^{t(\ell)}\frac{r(q/\ell^i)-r(q/\ell^{i-1})}{\phi(\ell^i)}
\]
where, for a divisor $d$ of $q$, $r(d)$ is the $\Z$-rank of the submodule of $M$ fixed by the subgroup of $\Rho$ of order $d$ and $\phi$ is Euler's totient function.
\end{lemma}


\begin{theorem}\label{th:bounds}
Let $M$ be a finitely generated $\Z[D]$-module and let $\ell$ be a prime. Then
\[
-L_\ell(M)\leq v_\ell(\rc{\Theta}{M})\leq U_\ell(M) 
\]
where $L_\ell(M) = U_\ell(M) = 0$ if $\ell \nmid q$ and 
\begin{align*}
L_\ell(M) &= 2v_\ell(\card{\tors{M}^D/q}) +2\rk{M^D}v_\ell(q) -v_\ell(\hq{\Rho}{M})\\
U_\ell(M) &= 2v_\ell(\card{\tors{M}^\Rho/q}) + 2\rk{M^\Rho}v_\ell(q) -v_\ell(\hq{\Rho}{M})
\end{align*}	
if $\ell \mid q$, with $v_\ell(\hq{\Rho}{M})$ given in Lemma \ref{lemma:rosen}.
\end{theorem}
\begin{proof}
The statement is a generalization of \cite[Corollary 3.15]{CaputoNuccio2020} whose proof carries over: we work out the details for the commodity of the reader. First note that $\rc{\Theta}{M}=1$ if $\ell \nmid q$ (see Remark \ref{rmk:dihedralrcval}). If, instead, $\ell\mid q$, from Theorem \ref{th:dihedralrccohom} we have
\[
v_\ell(\rc{\Theta}{M}) = v_\ell\left(\left(\frac{\hht{-1}{D}{M}}{\hht{0}{D}{M}}\right)^2\frac{\hht{0}{\Rho}{M}}{\hht{-1}{\Rho}{M}}\right)=2v_\ell\left(\frac{\hht{-1}{D}{M}}{\hht{0}{D}{M}}\right) +v_\ell(\hq{\Rho}{M}).
\]
Observe that $\HHT{0}{\Rho}{M}=M^\Rho/N_\Rho M$ is a quotient of $M^\Rho$ annihilated by $q$. We therefore have 
\begin{align*}
v_\ell\left(\frac{\hht{-1}{D}{M}}{\hht{0}{D}{M}}\right)&\leq
v_\ell\bigl(\hht{-1}{D}{M}\bigr)\\
&\leq v_\ell\bigl(\hht{-1}{\Rho}{M}\bigr) \quad \text{by \cite[Prop. 2.1]{CaputoNuccio2020}}\\
&= v_\ell\bigl(\hht{0}{\Rho}{M}\bigr) - v_\ell(\hq{\Rho}{M})\\ 
&\leq v_\ell(\card{\tors{M}^\Rho/q}) + \rk{M^\Rho}v_\ell(q) - v_\ell(\hq{\Rho}{M}).
\end{align*}
Similarly, $\HHT{0}{D}{M}=M^D/N_DM$ is a quotient of $M^D$ annihilated by $q$. Therefore
\begin{align*}
v_\ell\left(\frac{\hht{-1}{D}{M}}{\hht{0}{D}{M}}\right)&\geq -v_\ell\bigl(\hht{0}{D}{M}\bigr)\\
&\geq-v_\ell(\card{\tors{M}^D/q}) - \rk{M^D}v_\ell(q).\qedhere
\end{align*}
\end{proof}


\section{Arithmetic applications}\label{sec:applications}
In this section, we will apply the results of the previous section to various arithmetic contexts. We keep the notation of the previous section for $D, \Rho, \Sigma$ and $\Theta$ (see \eqref{eq:dgroup} and \eqref{eq:drelation}). Moreover, in this section $L/k$ denotes a Galois extension of number fields with Galois group isomorphic to $D$. We set $K=L^\Sigma$ and $F=L^\Rho$. 

\subsection{Ring of integers of number fields}
For any subgroup $H\leq D$, we denote by $\OO{L^H}{}$ the ring of integers of $L^H$. 

\begin{theorem}\label{th:rif}
Let $L/k$ be a Galois extension of number fields with Galois group isomorphic to $D$. We have
\[
\hht{i}{D}{\OO{L}{}} = \hht{j}{D}{\OO{L}{}} \quad \text{if $i+j\equiv -1 \pmod 4$}.
\]
\end{theorem}
\begin{proof}
First of all note that we only need to prove the assertion for $i,j \in \{-1, 0, 1, 2\}$ since the cohomology of $D$ is periodic of period $4$.
By a classical result of Fröhlich (see \cite[Theorem 7 (Additive)]{Frohlich1989}), the ring of integers is \textit{factor-equivalent} to a free $\Z[G]$-module. In the language of regulator constants, this is equivalent to $\rc{\Theta}{\OO{L}{}}=1$ (see \cite[Corollary 2.12]{Bartel2014}). Therefore Theorem \ref{th:dihedralrccohom} implies
\[
\hht{0}{\Theta}{\OO{L}{}}\hht{1}{\Theta}{\OO{L}{}}=\frac{\hht{0}{\Theta}{\OO{L}{}}}{\hht{-1}{\Theta}{\OO{L}{}}} = 1.  
\] 
Since $\Sigma$ and $\Rho$ are cyclic we further have
\begin{equation}\label{eq:yokoi}
\hht{i}{\Sigma}{\OO{L}{}}=\hht{j}{\Sigma}{\OO{L}{}} \quad \text{and} \quad \hht{i}{\Rho}{\OO{L}{}}=\hht{j}{\Rho}{\OO{L}{}}
\end{equation}
for any $i,j\in\Z$ (see \cite[Theorem 3]{Yokoi1964}). We deduce 
\begin{equation}\label{eq:h0h1ri}
\hht{0}{D}{\OO{L}{}}\hht{1}{D}{\OO{L}{}} = \hht{0}{\Rho}{\OO{L}{}}\hht{0}{\Sigma}{\OO{L}{}}^2 
\end{equation}
and
\begin{equation}\label{eq:h0h-1ri}
\hht{0}{D}{\OO{L}{}}=\hht{-1}{D}{\OO{L}{}}.
\end{equation}
We also have
\begin{equation}\label{eq:h0h2ri}
\hht{0}{D}{\OO{L}{}}\hht{2}{D}{\OO{L}{}} = \hht{0}{\Rho}{\OO{L}{}}\hht{0}{\Sigma}{\OO{L}{}}^2 
\end{equation}
thanks to \cite[Proposition 2.1 and Lemma 2.9]{CaputoNuccio2020} and \eqref{eq:yokoi}. Combining \eqref{eq:h0h-1ri}, \eqref{eq:h0h1ri} and \eqref{eq:h0h2ri} we conclude the proof.
\end{proof}

\begin{remark}
Yokoi conjectured that, for any $i,j\in\Z$, 
\[
\hht{i}{D}{\OO{L}{}}=\hht{j}{D}{\OO{L}{}}
\] 
(see \cite[p. 1]{Yokoi1964}). This conjecture was shown to be false by Lee and Madan (see \cite{LeeMadan1969}): they showed that, when $k=\Q$ and $L=\Q(\zeta_3, \sqrt[3]{2})$ with $\zeta_3$ a primitive $3$-rd root of unity, one has
\[
\hht{0}{D}{\OO{L}{}} = 3 \quad \text{and} \quad \hht{1}{D}{\OO{L}{}} = 1.
\]
Their proof is based on the analysis of the Galois action on an explicit integral basis of $\OO{L}{}$. 
Using \eqref{eq:h0h1ri}, we can give an alternative proof which can also be easily generalized to other radical extensions of the form $\Q(\zeta_3, \sqrt[3]{a})$ for $a\in \Q$. We first compute $\HHT{0}{\Rho}{\OO{L}{}} = \OO{F}{}/\tr{L/F}(\OO{L}{})$ where $\tr{L/F}$ is the trace of $L/F$. We know that $\tr{L/F}(\OO{L}{})$ is the smallest ideal of $F$ dividing the different $\diff{L/F}$ of $L/F$ (see \cite[Theorem 1]{Yokoi1960}). The different is easily determined once the discriminant $\disc{L/F}$ is known and the latter can be deduced from the Artin conductor $\cond{L/F}(\chi)$ of a non trivial character $\chi$ of $\Rho$, through the conductor-discriminant formula. By a result of Coleman-McCallum we have $v_\mathfrak{p}(\cond{L/F}(\chi))=2$, where $\mathfrak{p}$ is the prime of $F$ above $3$ and $v_\mathfrak{p}$ denotes the $\mathfrak{p}$-adic valuation (see \cite[Theorem 6.1]{ColemanMcCallum1988}). Therefore $v_\mathfrak{p}(\disc{L/F})=4$: in particular $\mathfrak{p}$ is totally ramified in $L/F$ and $v_\mathfrak{P}(\diff{L/F})=4$, where $\mathfrak{P}$ is the prime of $L$ above $\mathfrak{p}$. Thus $v_\mathfrak{p}(\tr{L/F}(\OO{L}{}))=1$ and 
\begin{equation}\label{eq:h0radical}
\hht{0}{\Rho}{\OO{L}{}} = 3.
\end{equation} 
Now note that $\hht{0}{\Sigma}{\OO{L}{}}=1$ since $L/K$ is tame. Therefore \eqref{eq:h0h1ri} reads
\begin{equation}\label{eq:h0h1rirmk}
\hht{0}{D}{\OO{L}{}}\hht{1}{D}{\OO{L}{}} = \hht{0}{\Rho}{\OO{L}{}}
\end{equation}
in this case. In particular, since $\hht{0}{\Rho}{\OO{L}{}}$ is not a square, we immediately see that $\hht{0}{D}{\OO{L}{}} \ne \hht{1}{D}{\OO{L}{}}$. We can in fact easily compute $\hht{0}{D}{\OO{L}{}}$ using the same idea as above: here the formula from differents in a tower of extensions gives $v_{\mathfrak{P}}(\diff{L/\Q})=7$, so that $v_3(\tr{L/\Q}(\OO{L}{}))=1$. Therefore $\hht{0}{D}{\OO{L}{}}=3$ and, by \eqref{eq:h0radical} and \eqref{eq:h0h1rirmk}, $\hht{1}{D}{\OO{L}{}}=1$. Thanks to Theorem \ref{th:rif}, we have
\[
\hht{i}{D}{\OO{L}{}} = \begin{cases}
1 & i \equiv 1,2 \pmod{4} \\
3 & i \equiv 0,3 \pmod{4}.
\end{cases}
\] 
\end{remark}

\subsection{Units of number fields}
We now focus on the case of units of number fields and retrieve some of the main results of \cite{CaputoNuccio2020} and \cite{Nuccio2025}. Let $S$ be a  Galois-stable finite set of places of $L$, containing the archimedean ones. For any subgroup $H\leq D$, we denote by $\OO{L^H}{S}$ the ring of $S(L^H)$-integers of $L^H$, where $S(L^H)$ is the set of places of $L^H$ which lie below places of $S$. Moreover, let $\cl{L^H}{S}$ denote the $S(L^H)$-class number of $L^H$, \ie the order of the $S(L^H)$-class group of $L^H$, and $\units{L^H}{S}$ the group of $S(L^H)$-units of $L^H$, which is a finitely generated $\Z$-module with 
\begin{equation}\label{eq:dirichletSunits}
\rk{\units{L^H}{S}} = \card{S(L^H)} - 1.
\end{equation}
For a place $v$ in $S(L^H)$, let $d_v(L/L^H)$ denote the local degree of $v$ in $L/L^H$ (\ie the order of the decomposition group of $v$ in $L/L^H$). 

The $\Rho$-Herbrand quotient of $S$-units is known explicitly:
\begin{equation}\label{eq:hqsunits}
\hq{\Rho}{\units{L}{S}}= \frac{1}{q}\prod_{w\in S(F)} d_w(L/F)
\end{equation}
(see \cite[Chapter IX, \S4, Corollary 2]{LangANT}).
 
\begin{theorem}[{\cite[Theorem 3.14]{CaputoNuccio2020}, \cite[Corollary 2.13]{Nuccio2025}}]\label{th:suf} 
Let $L/k$ be a Galois extension of number fields with Galois group isomorphic to $D$ and set $K=L^\Sigma$ and $F=L^\Rho$. Let $S$ be a Galois-stable finite set of places of $L$, containing the archimedean ones. Then
\[
\frac{\cl{L}{S}\left(\cl{k}{S}\right)^2}{\cl{F}{S}\left(\cl{K}{S}\right)^2} = \frac{\hht{0}{D}{\units{L}{S}}}{\hht{-1}{D}{\units{L}{S}}}\frac{\hht{-1}{\Sigma}{\units{L}{S}}}{\hht{0}{\Sigma}{\units{L}{S}}} \left(\prod_{v\in S_s(k)} d_v(L/k)\right)^{-1}
\] 
where $S_s(k)$ denotes the subset of $S(k)$ consisting of those places which split in $F/k$.
\end{theorem}
\begin{proof}
By \cite[Proposition 2.15]{Bartel2012} we have
\[
\rc{\Theta}{\units{L}{S}}\rc{\Theta}{\Z}^{-1}\rc{\Theta}{\Z[S]} = \prod_{H\leq D} (\reg{L^H}{S})^{2n_H}
\]
where $\reg{L^H}{S}$ is the regulator of the $S$-units of $L^H$. Since $\Theta$ is a Brauer relation, the formalism of the Artin $L$-functions gives a relation between Dedekind zeta functions:
\[
\prod_{H\leq D} \zeta_{L^H}^{n_H} = 1.
\] 
Using the formula for the residue of the Dedekind zeta function at $0$, we deduce that
\begin{equation}\label{eq:rccn} 
\rc{\Theta}{\units{L}{S}}\rc{\Theta}{\Z}^{-1}\rc{\Theta}{\Z[S]} = \prod_{H\leq D} \left(\frac{\card{\roots{L^H}}}{\cl{L^H}{S}}\right)^{2n_H}
\end{equation}
where $\roots{L^H}=\tors{\units{L^H}{S}}$ is the group of roots of unity of $L^H$. Now
\begin{equation}\label{eq:rootsinrelations}
\prod_{H\leq D} \card{\roots{L^H}}^{n_H}= 1
\end{equation}
(see for instance \cite[Lemma 3.17]{CaputoNuccio2020}) and 
\begin{align}
\rc{\Theta}{\Z} &= \prod_{H\leq G}\hht{0}{H}{\Z}^{-n_H} = q^{-1} \label{eq:rcz}\\
\rc{\Theta}{\Z[S]}&= \prod_{H\leq D} \left(\hht{0}{H}{\Z[S]}\right)^{-n_H} 
\end{align}
by Proposition \ref{prop:rcdual1st}, using that both $\Z$ and $\Z[S]$ are self-dual. To make $\rc{\Theta}{\Z[S]}$ more explicit, observe that $v_2(\rc{\Theta}{\Z[S]})=0$ by Remark \ref{rmk:dihedralrcval} and, if $\ell$ is an odd prime,
\begin{align*}
v_\ell(\rc{\Theta}{\Z[S]})&= v_\ell(\hht{0}{\Rho}{\Z[S]})-2v_\ell(\hht{0}{D}{\Z[S]}) \\
&= v_\ell(\hht{0}{\Rho}{\Z_\ell[S]})-2v_\ell(\hht{0}{D}{\Z_\ell[S]}) \quad \text{(by \eqref{eq:zlcohom})}\\
&= v_\ell(\hht{0}{D}{\Z_\ell[S]}) + v_\ell(\hht{2}{D}{\Z_\ell[S]}) -2v_\ell(\hht{0}{D}{\Z_\ell[S]}) \quad \text{(by \cite[Proposition 2.1]{CaputoNuccio2020})}\\
&= v_\ell(\hht{2}{D}{\Z_\ell[S]}) - v_\ell(\hht{0}{D}{\Z_\ell[S]}).
\end{align*}
Now $\Z_\ell[S]=\bigoplus_{v\in S(k)} \mathrm{Ind}_{H_v}^D\Z_\ell$, where $H_v$ is a decomposition group of $v\in S(k)$ in $D$, so that Shapiro's lemma implies 
\begin{align*}
\HHT{0}{D}{\Z_\ell[S]}& = \bigoplus_{v\in S(k)} \HHT{0}{H_v}{\Z_\ell} \cong \bigoplus_{v\in S(k)} \Z_\ell/\card{H_v}\\
\HHT{2}{D}{\Z_\ell[S]}& = \bigoplus_{v\in S(k)}  \HHT{2}{H_v}{\Z_\ell} \cong \bigoplus_{v\in S(k)}\HHT{1}{H_v}{\Q_\ell/\Z_\ell} = \bigoplus_{v\in S(k)}\pdual{(H_v^{ab})}\otimes_\Z\Z_\ell.
\end{align*}
Observe that $\card{H_v^{ab}}=2$ if $v$ does not split in $F/k$ (so in this case $\pdual{(H_v^{ab})}\otimes_\Z\Z_\ell$ is trivial) and $H_v^{ab}= H_v$ otherwise. 
We therefore have 
\begin{equation}\label{eq:rczs}
\rc{\Theta}{\Z[S]} = \left(\prod_{v \in S(k)\smallsetminus S_s(k)} \frac{d_v(L/k)}{2}\right)^{-1}.
\end{equation}
Then Theorem \ref{th:dihedralrccohom}, \eqref{eq:hqsunits}, \eqref{eq:rccn} and \eqref{eq:rootsinrelations} give
\begin{align*}
\left(\frac{\cl{L}{S}\left(\cl{k}{S}\right)^2}{\cl{F}{S}\left(\cl{K}{S}\right)^2}\right)^{2} &= \rc{\Theta}{\units{L}{S}}^{-1}\rc{\Theta}{\Z}\rc{\Theta}{\Z[S]}^{-1}\\
&=\left(\frac{\hht{0}{D}{\units{L}{S}}}{\hht{-1}{D}{\units{L}{S}}}\frac{\hht{-1}{\Sigma}{\units{L}{S}}}{\hht{0}{\Sigma}{\units{L}{S}}}\right)^2 \hq{\Rho}{\units{L}{S}}^{-1} q^{-1}\prod_{v\in S(k)\smallsetminus S_s(k)} \frac{d_v(L/k)}{2}\\
&=\left(\frac{\hht{0}{D}{\units{L}{S}}}{\hht{-1}{D}{\units{L}{S}}}\frac{\hht{-1}{\Sigma}{\units{L}{S}}}{\hht{0}{\Sigma}{\units{L}{S}}}\right)^2    \left(\prod_{v\in S_s(k)} d_v(L/k)\right)^{-2}.
\end{align*}
Here we have used that 
\begin{equation}\label{eq:localdegrees}
\prod_{w\in S(F)} d_w(L/F) = \left(\prod_{v\in S(k)\smallsetminus S_s(k)}\frac{d_v(L/k)}{2}\right)
\left(\prod_{v\in S_s(k)}d_v(L/k)\right)^2.\qedhere
\end{equation}
\end{proof}

Using Theorem \ref{th:bounds} one obtains bounds on the $S$-class number ratio (see \cite[Corollary 3.15]{CaputoNuccio2020}). For a number field $E$, we denote by $r_1(E)$ and $r_2(E)$ the number of real embeddings and conjugate pairs of complex embeddings of $E$, respectively.

\begin{corollary}\label{cor:cnbounds} 
In the same setting of Theorem \ref{th:suf}, we have, for a prime $\ell \mid q$,
\[
-A(L/k,S,\ell) \leq v_\ell\left(\frac{\cl{L}{S}\left(\cl{k}{S}\right)^2}{\cl{F}{S}\left(\cl{K}{S}\right)^2}\right) \leq B(L/k,S,\ell)
\]
where
\begin{align*}
A(L/k,S,\ell) &= \card{S(F)}v_\ell(q) + v_\ell(\card{\roots{F}/q}) - d_\ell(S(k))\\
B(L/k,S,\ell) &= (\card{S(k)}-1)v_\ell(q) + v_\ell(\card{\roots{k}/q}) - d_\ell(S_s(k))
\end{align*}
and
\begin{align*}
d_\ell(S_s(k)) &= \sum_{v\in S_s(k)} v_\ell(d_v(L/k))\\
d_\ell(S(k)) &= \sum_{v\in S(k)} v_\ell(d_v(L/k)).
\end{align*}
\end{corollary}
\begin{proof}
By \eqref{eq:rccn} and \eqref{eq:rootsinrelations} we need to find bounds for  
\[
-v_\ell(\rc{\Theta}{\units{L}{S}} + v_\ell(\rc{\Theta}{\Z}) - v_\ell(\rc{\Theta}{\Z[S]}).
\] 
To find bounds for the first of the above three terms we use Theorem \ref{th:bounds}, \eqref{eq:dirichletSunits} and \eqref{eq:hqsunits}. The second and third terms have closed form formulas thanks to \eqref{eq:rcz} and \eqref{eq:rczs}, respectively. Putting everything together and using \eqref{eq:localdegrees} we conclude.
\end{proof}

\begin{remark}
The bounds in Corollary \ref{cor:cnbounds} are slightly tighter than those of  \cite[Corollary 3.15]{CaputoNuccio2020}. In fact, for a number field $E$, let
$\defect{E}{q}$ be $0$ or $1$ depending on whether $\roots{E}/q$ is trivial or not.
In \cite[Corollary 3.15]{CaputoNuccio2020}, mainly for notational convenience, the terms $v_\ell(\card{\roots{k}/q})$ and $v_\ell(\card{\roots{F}/q})$ in Corollary \ref{cor:cnbounds} are replaced by $\defect{k}{q}v_\ell(q)$ and $\defect{F}{q}v_\ell(q)$, respectively, and one always has $\defect{E}{q}v_\ell(q) \geq v_\ell(\card{\roots{E}/q})$ (since $\roots{E}$ is cyclic).      	
\end{remark}

\begin{remark}
One can obtain a formula similar to that of Theorem \ref{th:suf} for $K$-groups of ring of integers (up to powers of $2$). In this case \eqref{eq:rccn} is replaced by \cite[(2.3)]{BarteldeSmit2013} and \eqref{eq:rootsinrelations} by \cite[Proposition 2.4]{BarteldeSmit2013}. Moreover the $\Rho$-Herbrand quotient of odd-numbered $K$-groups is known to be trivial, since $\Rho$ has odd order (see \cite[Remarque 4.3]{Kahn1993}). One can also deduce an analogue of the bounds of Corollary \ref{cor:cnbounds} in terms of the ranks of odd-numbered $K$-groups (which are known explicitly, see \cite[Theorem 6]{Weibel2005}). Applying the same idea to motivic cohomology groups, one can recover and generalize the bounds of \cite[Proposition 2.4]{zhou2015}. 
\end{remark}

\subsection{Elliptic curves}
Finally, we analyse the case of elliptic curves. For an elliptic curve $E$ defined over the number field $k$ and a subgroup $H\leq D$, $\rp{L^H}$ is the finitely generated $\Z$-module of $L^H$-rational points of $E$. Moreover we let $\tn{L^H}$ denote the product of suitably normalised Tamagawa numbers over all finite places of $L^H$ (see \cite[p. 571]{DokchitserDokchitser2010} for the exact definition of such normalisation) and $\ts{L^H}$ the Tate-Shafarevich group of $E/L^H$. 

Now let $q = \prod_\ell \ell^{t(\ell)}$ be the prime decomposition of $q$. Then Lemma \ref{lemma:rosen} in the context of elliptic curves reads
\begin{equation}\label{eq:hqec}
v_\ell(\hq{\Rho}{\rp{L}}) = t(\ell)\rpr{F} - \sum_{i=1}^{t(\ell)}\frac{\rpr{L(\ell^i)}-\rpr{L(\ell^{i-1})}}{\phi(\ell^i)}
\end{equation}
where, for a divisor $d$ of $q$, $L(d)/F$ is the subextension of $L/F$ of degree $d$.

\begin{theorem}\label{th:ecf} 
Let $L/k$ be a Galois extension of number fields with Galois group isomorphic to $D$ and set $K=L^\Sigma$ and $F=L^\Rho$. Let $E$ be an elliptic curve defined over $k$ and assume $\ts{L^H}$ is finite for all $H\leq D$. Then
\[
\prod_{H\leq D}\left(\frac{\card{\tors{\rp{L^H}}}^2}{\tn{L^H}\card{\ts{L^H}}}\right)^{n_H} =\left(\hht{0}{\Theta}{\rp{L}}\hht{1}{\Theta}{\rp{L}}\right)^{-1} =\frac{\hht{-1}{\Theta}{\rp{L}}}{\hht{0}{\Theta}{\rp{L}}}.
\] 
\end{theorem}
\begin{proof}
As in the beginning of the proof of Theorem \ref{th:suf}, one obtains a formula 
\[
\prod_{H\leq D}\left(\frac{\card{\tors{\rp{L^H}}}^2}{\tn{L^H}\card{\ts{L^H}}}\right)^{n_H} = \rc{\Theta}{\rp{L}}
\]	
(see \cite[(2.5) and (2.6)]{BarteldeSmit2013}). The result then follows by Theorem \ref{th:rcformula}.
\end{proof}

As an immediate consequence of the above theorem we obtain the following parity result (a particular case of which is contained in \cite[Proposition 4.17]{DokchitserDokchitser2010}).

\begin{corollary}\label{cor:parity}
In the same setting as Theorem \ref{th:ecf}, we have, for a prime $\ell\mid q$,
\[
v_\ell(\tn{F})- v_\ell(\tn{L}) \equiv v_\ell(\hq{\Rho}{\rp{L}}) \pmod 2
\]
and $v_\ell(\hq{\Rho}{\rp{L}})$ is given by \eqref{eq:hqec}.
\end{corollary} 
\begin{proof}
Observe that
\begin{align*}
\prod_{H\leq D}\left(\frac{\card{\tors{\rp{L^H}}}^2}{\tn{L^H}\card{\ts{L^H}}}\right)^{n_H} &\equiv \prod_{H\leq D}\left(\tn{L^H}\right)^{-n_H} \pmod{(\Q^\times)^2}\\
&\equiv \frac{\tn{F}}{\tn{L}} \pmod{(\Q^\times)^2}
\end{align*}
since Tate-Shafarevich groups (assumed to be finite) have square order (see \cite[Chapter X, Theorem 4.14]{Silverman2009}) and $n_\Sigma = n_{\{\mathrm{id}\}}=2$. Similarly we have
\[
\prod_{H\leq D}\left(\frac{\hht{-1}{H}{\rp{L}}}{\hht{0}{H}{\rp{L}}}\right)^{n_H} \equiv \frac{\hht{0}{\Rho}{\rp{L}}}{\hht{-1}{\Rho}{\rp{L}}} \pmod{(\Q^\times)^2}.\qedhere
\]
\end{proof}

\begin{remark}
When $q$ is a prime, we have
\begin{align*}
\frac{q \rpr{F} -\rpr{L}}{q-1} &\equiv v_q(\tn{F})- v_q(\tn{L})  \pmod 2  \quad \text{by Cor. \ref{cor:parity}}\\
\frac{2\left(\rpr{k} - \rpr{K}\right)}{q -1}-\rpr{F} &\equiv v_q(\tn{F})- v_q(\tn{L}) \pmod 2 \quad \text{by \cite[Prop. 4.17]{DokchitserDokchitser2010}}.
\end{align*}
These two congruences are indeed the same since
\[
2(\rpr{k} - \rpr{K}) = \rpr{F} - \rpr{L}
\]
by \cite[p. 572]{DokchitserDokchitser2010}.
\end{remark}

Using Theorem \ref{th:bounds}, one immediately obtains the following bounds.

\begin{corollary}
In the same setting as Theorem \ref{th:ecf}, we have, for a prime $\ell\mid q$,
\[
-L_\ell(E) \leq v_\ell\left(\prod_{H\leq D}\left(\frac{\card{\tors{\rp{L^H}}}^2}{\tn{L^H}\card{\ts{L^H}}}\right)^{n_H}\right)\leq U_\ell(E)
\] 
where 
\begin{align*}
L_\ell(E)& = 2v_\ell(\tors{\rp{k}}/q) + 2\rpr{k}v_\ell(q) - v_\ell(\hq{\Rho}{\rp{L}})\\
U_\ell(E)& = 2v_\ell(\tors{\rp{F}}/q) + 2\rpr{F}v_\ell(q) - v_\ell(\hq{\Rho}{\rp{L}})
\end{align*}
and $v_\ell(\hq{\Rho}{\rp{L}})$ is given by \eqref{eq:hqec}.
\end{corollary}


\begin{acknowledgment}
We are very grateful to the referee for their careful reading of the manuscript and their useful corrections and comments, which helped us improve the clarity of our paper.
\end{acknowledgment}

\printbibliography
\end{document}